\documentclass{amsart}
\usepackage{enumerate,mathrsfs,verbatim}

\title{An elementary way to introduce a Perron-like integral}
\author{Hana Bendov\'a and Jan Mal\'y}

\thanks{%
The second author was supported in part 
he research project MSM 0021620839 
of the Czech Ministry M\v SMT
and in part by the grant GA\,\v{C}R 201/09/0067}

\address{Department of Mathematical Analysis, Charles University,
So\-ko\-lovsk\'a 83, 186~00 Prague 8, Czech Republic}
\email{\tt haaanja@gmail.com}
\address{Department of Mathematical Analysis, Charles University,
So\-ko\-lovsk\'a 83, 186~00 Prague 8, Czech Republic, and\newline
\hglue\parindent Department of Mathematics, 
J. E. Purkyn\v e University,
\v Cesk\'e ml\'a\-de\v ze~8,
400~96 \'Ust\'{\i} nad Labem, Czech Republic}
\email{\tt maly@karlin.mff.cuni.cz}

\def\cit#1{}
\def\diam{\operatorname{diam}}
\def\ep{\varepsilon}
\def\er{\mathbb R}
\def\bF{\mathbf F}
\def\bG{\mathbf G}
\def\bPhi{\mathbf \Phi}
\def\rn{\mathbb R^n}

\def\D{\mathcal D}
\def\I{\mathcal I}
\newtheorem{theorem}{Theorem}
\newtheorem{proposition}{Proposition}
\newtheorem{lemma}{Lemma}
\newtheorem{corollary}{Corollary}
\theoremstyle{definition}
\newtheorem{definition}{Definition}

\newtheorem{remark}{Remark}
\def\eqn#1$$#2$${\begin{equation}\label#1#2\end{equation}}

\begin{document}

\begin{abstract} We give an alternative definition of integral at the
generality of the Perron integral and propose an exposition of 
the foundations of integral theory starting from this new definition. 
Both definition and proofs needed for the development are unexpectedly simple.
We show how to adapt the definition to cover the multidimensional and 
Stieltjes case and prove that our integral is equivalent to the 
Henstock-Kurzweil(-Stieltjes) integral.
\end{abstract}



\maketitle

\section{Introduction}

The aim of its paper is to built  self-contained foundations of the 
theory of nonabsolutely convergent integral based on a new definition.
Our definition is a slight modification of definitions used previously, but
provides a possibility of a surprisingly comprehensible development 
of the theory.

We are focused on integrals which include the Lebesgue integral and 
integrate all derivatives. First such a construction was done by
Denjoy \cite{De} in 1912, shortly followed by Luzin \cite{Lu}.
The integral of Perron \cite{Pe} from 1914 uses families of major and minor
functions instead of a single antiderivative.
A ``weighted'' analogue of the Perron integral is the Perron-Stieltjes
integral introduced by Ward \cite{Wa}.
In 1957, Kurzweil \cite{Ku}
introduced a gauge generalized Riemann type integral, which is equivalent to the
Perron integral.
The same construction was found
independently by Henstock \cite{He},
see also \cite{He1}, \cite{He2}. 
The advantage of this construction is
that it is based on Riemann sums which are commonly used to illustrate and motivate
the concept of integral. There have been made serious attempts to
build an elementary course of integration on basis of the Henstock-Kurzweil
integral, e.g.\ \cite{LV}. 
A completely different idea of a curiously simple definition of integral 
in Perron-like generality is due to Tolstov \cite{To}.

We present a definition of integral which is also equivalent with 
Perron's definition.
An intermediate step 
between our integral and the Perron integral is the variational integral.
This has been introduced by Henstock \cite{He-var} and admits various
formulations, see e.g.\ \cite{He-var1}. 
In \cite{Go}, Definition 11.7, we may find a version which can be
stated as follows:

A function $f:[a,b]\to\er$ is variational integrable if there
exists a function $F:[a,b]\to\er$ with the following property: For each $\ep>0$
there exists an increasing function $\varphi_{\ep}:[a,b]\to\er$ 
and a strictly positive function $\delta:[a,b]\to\er$
such that
$[\varphi_\ep]_a^b<\ep$ and for each $x,y\in [a,b]$ we have
$$
|y-x|\le \delta(x)\implies |F(y)-F(x)-f(x)(y-x)|\le
|\varphi_{\ep}(y)-\varphi_{\ep}(x)|.
$$
The integral of $f$ over $[a,b]$ is then $[F]_a^b$.

We simplify further the definition of variational integral: we use 
a single control function and replace the explicit 
description of $\ep$--$\delta$ dependence
by an ordinary limit.
This enables, among others, to use a language
in which the definition looks almost like 
the ordinary definition of (anti)derivative and proofs of tools like 
integration by parts and change of variables are short and elegant.

The idea of a single control function appears in analysis also in other contexts:
Cornea \cite{Cor1}, \cite{Cor2} (see also \cite{LMNS})
uses a control function to modify Perron's 
construction of solution of the Dirichlet problem 
in potential theory. Notice that there is an parallel with our construction;
also in this case, the original Perron's idea is based on upper and lower
functions. 
Another relevant concept is that of delta-convex mappings. Originally,
(scalar) delta-convex functions are diferences of convex functions.
Vesel\'y and Zaj\'\i \v cek \cite{VZ} use control functions to 
generalize delta-convexity to the vector valued case.

Now, our definition is the following:

\begin{definition} Let $I=(a,b)\subset \er$ be an interval and $f,F:I\to\er$ be functions.
We say that $f$ is an \textit{$MC$-derivative}
(monotonically controlled derivative)
of $F$ if there exists a strictly increasing
function  $\varphi\colon I\to \er$ (the so-called \textit{control function} to
the pair $(F,f)$)
such that 
\eqn{def}
$$
\lim_{y \to x}{\frac{F(y)-F(x)-f(x)(y-x)}{\varphi(y)-\varphi(x)}}=0,\qquad x\in
I.
$$
We also say that $F$ is an \textit{$MC$-antiderivative} (or an \textit{indefinite $MC$-integral})
of $f$.
\end{definition}

Recall that the $MC$-integral coincides with the Perron integral
for which the development of calculus is well known.
Despite of this we 
hope that it is valuable to present an independent development of the theory. Indeed, we believe that the concept 
of $MC$-integral is comprehensible for students-beginners. Therefore we want to indicate how the theory of integral
can be developed from scratch. However, we address this exposition to experienced mathematicians,
so that our text is not exactly in the style of a course for beginners; this will be a task for a textbook project.

Any reasonable notion of indefinite integral must have the property that
two indefinite integrals of the same function 
can differ only by an additive constant. This is mostly established by the
observation that if $F'$ is positive, the $F$ is increasing.
This property is valid also for $MC$-differentiation as shown in Section 
\ref{s:basic}.

Then our exposition includes basic tools for integration 
like integration by parts and change of variables, 
this will be in Section \ref{s:calc}. 
The results are well known in this generality, see e.g.\ \cite{Pf-pp},
\cite{Ku-book}.
However, our definition leads to proofs which are very simple.

In Section \ref{s:mct} we prove the monotone convergence theorem. 
This opens a gateway for a development of the theory
of integral in spirit of courses of Lebesgue integration.

The construction of the $MC$-integral can be easily adapted to 
more general situations.
We show a simultaneous generalization to  functions of several variables and to
Stieltjes integration (in another terminology, integration with respect to Radon
measures), this is done in Section \ref{s:multdim}. 
The main result of this section is that, even in this generality,
the $MC$-integral coincides with a corresponding integral defined by the
Henstock-Kurzweil construction (and thus also with the Perron integral). 
This result is not so hard once we know that
the Henstock-Kurzweil integral coincides with the variational integral
\cite{He1}; the crucial step of this equivalence is the Henstock lemma.
We demonstrate the correspondence on a model case, but the idea indicates
that practically each integral constructed via gage-fine tagged 
partitions has its $MC$-version and vice versa.
For various such
definition of multidimensional 
integrals, discussion of problems  and further bibliography  we
refer e.g.\ to \cite{He1}, \cite{He2}, \cite{JK}, \cite{JKS}, \cite{Ju},
\cite{Ku-book},  \cite{Ma1}, \cite{NP}, \cite{Pf1}, \cite{Pf}.

We will not develop the foundations of the multidimensional theory as 
in the first part of the paper, because here the advange in simplicity
is already not so distinct.
However, it is worth to mention that once started the
development of integration theory with the $MC$-integral, it is possible to
proceed to multidimensional integration and obtain the same results
as in the theory of multidimensional Henstock-Kurzweil(-Stieltjes) 
integral. From the didactical point of view, it is perhaps recommendable 
to resctrict 
soon the attention to the class of absolutely integrable functions, which
are exactly the Lebesgue(-Stieljes) integrable functions.

In this paper, \textit{positive} 
means $\ge 0$ whereas $>0$ is labelled as \textit{strictly positive}.
Similar convention applies to the terms \textit{increasing} and \textit{strictly
increasing}.

\section{$MC$-derivatives}\label{s:basic}

In this section we prove some basic properties of $MC$-differentiation.
First, we note that pairs $(F,f)$ such that $f$ is a $MC$-derivative of $F$ form 
a vector space. Also, it is evident that any ordinary derivative 
is an $MC$-derivative. However, there is one serious difference. Ordinary
derivatives are unique. If we want to have a concept of derivative general
enough to differentiate any indefinite Lebesgue integral, we necesarilly lose
uniqueness, namely, the derivative is pointwise determined only up to a set 
of measure zero. At this stage of exposition we do not need to speak on sets of
measure zero, however, it may be useful to note that the exceptional sets are
small. To illustrate this phenomenon, we asume that $f$ and $g$ are
$MC$-derivatives of $F$, with control functions $\varphi$ and $\psi$,
respectively. Then it is easy to observe that the monotone function
$\eta=\varphi+\psi$ has infinite derivative at each point of the set
$\{x:f(x)\ne g(x)\}$. 

It is useful to notice that if we add an increasing function to a control
function to $(F,f)$, we obtain also a control function to $(F,f)$.

If $\varphi$ is a control function to $(F,f)$, then any function of the form
$\alpha\varphi+\beta$, where $\alpha,\,\beta$ are constants, $\alpha>0$, is
also a control function to $(F,f)$. Such a modification of a control function is
called a \textit{rescaling}.

\begin{proposition}\label{p:cont}
Let $F$ be an indefinite $MC$-integral of a function $f$ on an interval
$I\subset\er$. 
Then $F$ is continuous.
\end{proposition}

\begin{proof}
Let $x\in I$. Since $\varphi$ is locally bounded, from \eqref{def} we obtain
$$
\lim_{y\to x}(F(y)-F(x)-f(x)(y-x))=0.
$$
It follows that $F$ is continuous at $x$.
\end{proof}

In the following theorem we prove that an indefinite $MC$-integral of a
positive function is increasing. 

\begin{theorem}\label{t:mp} 
Let $F$ be an indefinite $MC$-integral of a 
function $f\ge 0$ on an open interval $I=(a_0,b_0)\subset \er$. 
Then $F$ is increasing.
\end{theorem}

\begin{proof} Suppose that there exist $a,b\in I$ such that $a<b$ and
$F(b)-F(a)$. By rescaling we find a control function $\varphi$ with
$(F+\varphi)(a)>(F+\varphi)(b)$. We denote $G=F+\varphi$.
We set $a_1=a$, $b_1=a$ and
$c_1=\frac12(a_1+b_1)$. We choose $[a_2,b_2]$ among the intervals 
$[a_1,c_1]$, $[c_1,b_1]$ such that $G(a_2)>G(b_2)$. We continue recursively and
construct a nested sequence of closed intervals $[a_k,b_k]$ such that
$b_k-a_k=2^{-k}(b-a)$ and $G(a_k)>G(b_k)$.
There exists a point $x$ in the intersection of all intervals 
$[a_k,b_k]$. For each $k$ we can choose $x_k$ among the points $a_k$, $b_k$
such that $x_k\ne x$ and $G(x_k)-G(x)$ have the opposite sign to $x_k-x$.
Since 
$$
\lim_{y\to x}\frac{F(y)-F(x)-f(x)(y-x)}{\varphi (y)-\varphi (x)}=0,
$$
we obtain
$$
\lim_{k\to\infty}\frac{G(x_k)-G(x)-f(x)(x_k-x)}{\varphi(x_k)-\varphi(x)}=1,
$$
however, 
$$
\frac{G(x_k)-G(x)}{\varphi(x_k)-\varphi(x)}<0 \quad \text{and}\quad
\frac{f(x)(x_k-x)}{\varphi(x_k)-\varphi(x)}\ge0,\qquad k=1,2,\dots
$$
This is a contradiction.
\end{proof}

\begin{corollary}\label{c:byconstant} 
If $F$, $G$ are indefinite $MC$-integrals of a function $f$,
then $F-G$ is constant.
\end{corollary}

\begin{proof} The function $F-G$ is an indefinite $MC$-integral of the zero
function and thus it is both increasing and decreasing.
\end{proof}

\section{Calculus of $MC$-integral}\label{s:calc}

\begin{definition} If $f$ has an indefinite $MC$-integral $F$ on $(a,b)$
and $F$ has one-sided proper limits $F(a_+)$, $F(b_-)$ at the endpoints,
then the (definite) $MC$-integral of $f$ over $(a,b)$
is defined as the increment of $F$:
$$
\int_a^bf(x)\,dx=[F]_a^b,
$$
where $[F]_a^b$ denotes $F(b_-)-F(a_+)$. By Corollary \ref{c:byconstant}, this
definition is correct, namely, it does not depend on the choice of the
indefinite integral. It is obvious that integral is a linear functional.
One could also define extended-real-valued integrals this way, but our convention
will be that all integrals are real.
Notice that in this text, all integral symbols refer
to $MC$-integration unless specified otherwise.
\end{definition}

\begin{proposition}
Suppose that $f$ is an $MC$-derivative of $F$ and $g$ is an $MC$-derivative of
$G$ on $I=(a,b)$. 
Then $fG+Fg$ is an $MC$-derivative of $FG$. Hence the formula on
integration by parts
$$
\int_a^bf(x)G(x)\,dx=[FG]_a^b-\int_a^bF(x)g(x)\,dx
$$
holds if the increment and the integral on the right are well defined.
\end{proposition}

\begin{proof}
Let $\varphi$ control the pair $(F,f)$ and $\psi$ control the pair $(G,g)$.
We have
\eqn{prod}
$$
\aligned
&F(y)G(y)-F(x)G(x)-(f(x)G(x)+F(x)g(x))(x-a)
\\&\quad =F(y)\Bigl(G(y)-G(x)-g(x)(y-x)\Bigr)
\\&\qquad+G(x)\Bigl(F(y)-F(x)-f(x)(y-x)\Bigr) 
\\&\qquad+g(x)(y-x)\Bigl(F(y)-F(x)-f(x)(y-x)\Bigr)
\\&\qquad+f(x)g(x)(y-x)^2,\qquad x,y\in I.
\endaligned
$$
By Proposition \ref{p:cont}, $F$ is continuous. 
It is then easily seen from \eqref{prod} that the pair $(FG,\,fG+Fg)$ is
controlled by $\eta(x)=\varphi(x)+\psi(x)+x$. Now, by the assumptions, there
exists an indefinite $MC$-integral $H$ of $Fg$ with a well defined increment.
It follows that $FG-H$ is an indefinite $MC$-integral of $fG$ with a well
defined increment.
\end{proof}

\begin{proposition}
Suppose that $F$ is a strictly increasing function 
which maps open interval $(a,b)$
onto an open interval $(c,d)$. Let $G:(c,d)\to \er$ be a function.
Let $f$ be an $MC$-derivative of $F$ on $(a,b)$ and $g$ be an $MC$-derivative 
of $G$ on $(c,d)$. Then $x\mapsto g(F(x))f(x)$ is an $MC$-derivative of
$G\circ F$ on $(a,b)$. Hence, the formula on change of variables
$$
\int_c^dg(y)\,dy=\int_a^bg(F(x))f(x)\,dx
$$
holds provided that the above assumptions are satisfied and at least one of the
integrals converges.
\end{proposition}

\begin{proof}
Let $\varphi$ control the pair $(F,f)$ and $\psi$ control the pair $(G,g)$.
We have
$$
\aligned
&G(F(x'))-G(F(x))-g(F(x))f(x)(x'-x)
\\&\quad=G(F(x'))-G(F(x))-g(F(x))(F(x')-F(x))
\\&\qquad
+g(F(x))(F(x'))-F(x)-f(x)(x'-x)),\qquad x,x'\in (a,b).
\endaligned
$$
Since
$$
\aligned
&\lim_{x'\to
x}\frac{G(F(x'))-G(F(x))-g(F(x))(F(x')-F(x))}{\psi(F(x'))-\psi(F(x))}
\\&\quad
=\lim_{y'\to y}\frac{G(y')-G(y)-g(y)(y'-y)}{\psi(y')-\psi(y)}=0,
\qquad x\in (a,b),\;y=F(x),
\endaligned
$$
we easily infer that $\psi\circ F+\varphi$ controls the pair
$(G\circ F,\;(g\circ F)\;f)$. The statement concerning integration follows
immediately.
\end{proof}	

\begin{lemma}\label{l:bounded}
Let $F,\,f:(a,b)\to\er$ be functions and $(a_k)_k$,
$(b_k)_k$ be sequences of real numbers. Suppose that
$a_k\searrow a$ and $b_k\nearrow b$. 
If $f$ is an $MC$-derivative of $F$ on each $(a_k,b_k)$, then $f$ is an
$MC$-derivative on $F$ on $(a,b)$. Moreover, the control function on $(a,b)$
can be chosen to be bounded.
\end{lemma}

\begin{proof} We may assume that $(a_k)_k$ is strictly decreasing and 
$(b_k)_k$ is strictly decreasing. For each $k$, let 
$\varphi_k$ be a control function to $(F,f)$ on $(a_{k+1},b_{k+1})$.
Then $\varphi_k$ is bounded on $(a_k,b_k)$ and by a rescaling we may assume
that $0<\varphi_k<1$ on $(a_k,b_k)$. 
Set
$$
\psi_k(x)=
\begin{cases}
0,& x\le a_k,\\
\varphi_k(x), & a_k<x< b_k,\\
1,&x\ge  b_k.
\end{cases}
$$
Then the function 
$$
\varphi=\sum_{k=1}^{\infty}2^{-k}\psi_k
$$
is obviously a bounded control function to $(F,f)$ on $(a,b)$.
\end{proof}

\begin{proposition}\label{p:glue}
Let $a,b,c\in\er$, $a<b<c$. Let $f:(a,c)\to\er$ 
be $MC$-integrable on $(a,b)$ and $(b,c)$.
Then $f$ is $MC$-integrable on $(a,c)$ and
$$
\int_a^cf(x)\,dx=\int_a^bf(x)\,dx+\int_b^cf(x)\,dx.
$$
\end{proposition}

\begin{proof}
Let $F_1$ be an indefinite $MC$-integral of $f$ on $(a,b)$ 
controlled by $\varphi_1$ and
$F_2$ be an indefinite $MC$-integral of $f$ on $(b,c)$ controlled by
$\varphi_2$. By Lemma \ref{l:bounded} we can assume that $\varphi_1$, $\varphi_2$ 
are bounded. 
Then the 
one-sided proper limits $F_1(b_-)$, $\varphi_1(b_-)$, $F_2(b_+)$ 
and $\varphi_2(b_+)$ exist. By adding suitable constants to $F_i$ and 
rescaling $\varphi_i$, $i=1,2$,
we can arrange that $F_1(b_-)=F_2(b_+)=0$ and
$\varphi_1(b_-)<0<\varphi_2(b_+)$.
We set
$$
F(x)=
\begin{cases}
F_1(x),& x\in (a,b),\\
0, & x=b,\\
F_2(x),&x\in (b,c),
\end{cases}
\qquad 
\varphi(x)=
\begin{cases}
\varphi_1(x),& x\in (a,b),\\
0, & x=b,\\
\varphi_2(x),&x\in (b,c)\;.
\end{cases}
$$
We claim that the function $F$
is an indefinite $MC$-integral of $f$ on $(a,c)$ controlled by
$\varphi$. Indeed, \eqref{def} holds obviously at each
$x\in (a,b)\cup(b,c)$. For $x=b$ we use the jump of $\varphi$ at $b$
to observe that the limit in \eqref{def} reduces to
$$
\lim_{y\to x}\bigl(F(y)-F(x)-f(x)(y-x)\bigr).
$$
But the last limit clearly vanishes by the continuity of $F$.

\end{proof}

\section{Monotone convergence theorem}\label{s:mct}

In this section we establish the monotone convergence theorem
for the $MC$-integral. This can be applied to show that
the $MC$-integral includes the Lebesgue integral. Namely,
constants are integrable over bounded intervals. 
Using Proposition \ref{p:glue} we obtain that all step 
(=piecewise constant) functions are integrable. We can define
measurable sets as those sets $M$, for which the characteristic function
$\chi_M$ has an indefinite $MC$-integral. 
It is well known that a system of sets which contains all intervals and
is closed under monotone unions and intersections contains already
all Borel sets, see e.g.\ \cite{As}, 1.3.9.
Alternatively we can use Dynkin systems, see e.g.\ \cite{Co}, Section 1.6.
This is the step in which the monotone convergence theorem below is needed.

If we define the measure of a measurable $M$ as the integral of the characteristic
function of $M$ (or as $\infty$ if this integral diverges),
we observe that the ``measure'' is complete (all subsets of null sets are
measurable) and thus the class of all measurable sets contains 
all Lebesgue measurable sets. This argumentation not only leads to a proof that
the $MC$-integral includes the Lebesgue integral, but also bypasses some
difficult steps in construction of the Lebesgue measure.

\begin{theorem}[Monotone convergence theorem]
Let $I=(a,b)$ be an open interval and $(f_k)_k$ be a sequence of 
$MC$-integrable functions on $I$,
$f_n\nearrow f$. If 
\eqn{finite}
$$
\lim_{k\to\infty}\int_a^bf_k(x)\,dx<+\infty,
$$
then $f$ is $MC$-integrable over $I$ and
$$
\int_a^bf(x)\,dx=\lim_{k\to\infty}\int_a^bf_k(x)\,dx.
$$
\end{theorem}

\begin{proof} By subtracting $f_1$ we may achieve that $f_k\ge 0$,
$k=1,2,\dots$.
For each $k$, let $F_k$ be the indefinite $MC$-integral of $f_k$
normalized by $F_k(a_-)=0$. Then $(F_k)_k$ is an increasing sequence of increasing 
functions and we can define $F=\lim_k F_k$. From \eqref{finite} we infer that
$F$ is bounded in $(a,b)$, also it is easy to observe that
$F(b_-)=\lim_{k\to\infty}F_k(b_-)$.
Replacing, if
necessary, $(F_k)_k$ by a subsequence, we may assume that
\eqn{rychle}
$$
F_k(b_-)>F(b_-)-2^{-k}.
$$
Since $f_j-f_k\ge 0$ for $j>k$, the function $F_j-F_k$
is increasing by Theorem \ref{t:mp}. Passing to the limit we obtain that
each $F-F_k$ is increasing. For each $k$, let $\varphi_k$ be a control function
to $(F_k,f_k)$. By Lemma \ref{l:bounded}
we may assume that $\varphi_k$ is bounded and thus it may be rescaled to satisfy 
$0<\varphi_k<1$. We set
$$
\varphi(x)=\sum_{k=1}^{\infty}2^{-k}\varphi_k(x)+
\sum_{k=1}^{\infty}k(F(x)-F_k(x))+x.
$$
From \eqref{rychle} we infer that $\varphi$ is finite in $(a,b)$; obviously
it is  strictly increasing. We claim that $\varphi$
controls $(F,f)$. We choose $x\in I$ and $\ep>0$.
We find an integer $k>0$ such that $\frac1k<\ep$ and 
$f(x)-f_k(x)<\ep$. Then we estimate
$$
\aligned
F(y)-F(x)&-f(x)(y-x)=
\Bigl(F_k(y)-F_k(x)-f_k(x)(y-x)\Bigr)
\\&\quad+ \Bigl(F(y)-F(x)-\bigl( F_k(y)-F_k(x)\bigr)\Bigr)
\\&\quad+  (f_k(x)-f(x))(y-x)
\\& \le \bigl(F_k(y)-F_k(x)-f_k(x)(y-x)\bigr)
+\bigl(\tfrac
1k+\ep\bigr)(\varphi(y)-\varphi(x)).
\endaligned
$$
Therefore
$$
\limsup_{y\to x}\frac{F(y)-F(x)-f(x)(y-x)}{\varphi(y)-\varphi(x)}\le 2\ep,
$$
which verifies the claim.
Hence $F$ is an indefinite $MC$-integral of $f$ and 
$$
\int_a^bf(x)\,dx=[F]_a^b=\lim_{k\to\infty}[F_k]_a^b
=\lim_{k\to\infty}\int_a^bf_k(x)\,dx.
$$
\end{proof}

\section{A comparison with the Henstock-Kurzweil integral}
\label{s:multdim}

In this section we show that our $MC$-integral coincides with 
the Henstock-Kurzweil integral. This will be done in the framework
of multidimensional Stieltjes integration.
 
The multidimensional integration requires the language of 
interval functions.
For an introduction to manipulation with interval functions, in particular
to their differentiation, we refer to Saks \cite{Sa}. 
By interval in $\rn$ we mean a Cartesian product of one-dimensional intervals.
We denote by $\I$ the collection of all nondegenerate 
bounded closed intervals in $\rn$.
A finite set $\D\subset \I$ is called a 
\textit{partition} of and interval $I\in\I$ if the intervals 
from $\D$ are nonoverlapping (i.e.\
have disjoint interiors) and $\bigcup_{Q\in\D}Q=I$. 
A function $\bF:\I\to \er$ is said to be 
\begin{itemize}
\item \textit{additive}, if for each interval $I\in\I$
and each partition $\D$ of $I$ we have
$$
 \sum_{Q\in\D}\bF(Q)=\bF(I),
$$
\item \textit{superadditive}, if for each interval $I\in\I$
and each partition $\D$ of $I$ we have
$$
 \sum_{Q\in\D}\bF(Q)\le\bF(I),
$$
\end{itemize}

There are many possibilities how to modify the definition below, for example
to require some ``regularity'' of intervals in the limiting process
like in \cite{Ma1}. This will yield a variety of non-equivalent integrals. 
We illustrate our approach on the simplest model case.
We consider only indefinite integrals. The definite integrals over 
intervals $\notin \I$ can be defined by an appropriate limit process.

\begin{definition} Let $f:\rn\to\er$ be a function and $\bF:\I\to\er$, 
$\bG:\I\to\er$ be additive interval functions.
We say that  $f$ is an \textit{$MC$-derivative} of $\bF$ with respect to $\bG$,
or that $\bF$ is an \textit{indefinite $MC$-integral} of $f$ 
with respect to $\bG$,
if there exists a  superadditive interval function $\bPhi>0$  
(called a \textit{control function}) such that for each $x\in\rn$
and for each sequence $(Q_k)_k$ of intervals from $\I$ such 
that $x\in\bigcap_kQ_k$ and $\diam Q_k\to0$ we have
\eqn{mult}
$$
\lim_{k\to\infty}\frac{\bF(Q_k)-f(x)\bG(Q_k)}{\bPhi(Q_k)}=0.
$$
If in the definition of $MC$-derivative ($MC$-integral) we require $\bPhi$
to be additive, we denote the result as $AMC$-derivative ($AMC$-integral).
\end{definition}

Each  superadditive interval function $\bPhi>0$ has the property that
$$
P,Q\in\I,\quad P\subset Q\implies \bPhi(P)\le \bPhi(Q).
$$
Therefore the terminology ``monotonically controlled'' is again reasonable.

\begin{remark}
Additive interval functions in $\er$ have the form
$$
\bG([a,b])=G(b)-G(a),
$$
where $G:\er\to\er$ is an ``ordinary'' function.
For the one-dimensional Stieltjes differentiation and integration,
\eqref{mult} reduces to 
$$
\lim_{y\to x}\frac{F(y)-F(x)-f(x)(G(y)-G(x))}{\varphi(y)-\varphi(x)}=0.
$$
There is no need to use superadditive control functions in $\er$ because 
each superadditive function $\bPhi>0$ is easily majorized by an additive function,
using the increasing function
$$
\varphi(x)=
\begin{cases}
\bPhi([0,x]),&x>0,\\
0,&x=0,\\
-\bPhi([x,0]),&x<0.
\end{cases}
$$
In higher dimension, the relation between the $MC$ and $AMC$ definition
is not so clear, see \cite{He-var1}.
Since the $MC$-integral includes the $AMC$-integral and is more easy to handle,
we prefer $MC$-integration.
On the other hand, the notion of additive functions may seem to be
more elementary and for the purpose of absolute integration the 
concept of $AMC$-integration is sufficient.
\end{remark}

We recall the definition of $H\!K$-integral with respect to an additive interval
function $\bG$
as it is defined e.g.\ in \cite{Ku-book}.

\begin{definition}
Let $I\in\I$ be an interval. 
A \textit{tagged partition} of $I$ is defined as a couple $(\D,\tau)$ where 
$\D$ is a partition of $I$ and $\tau:\D\to\rn$ is a mapping such that 
\eqn{inner}
$$
\tau(Q)\in Q,\qquad Q\in\D.
$$
The condition \eqref{inner} is not always required in literature 
(it should be dropped for McShane integration), but it should be assumed
for the purpose of $H\!K$-integration. 
We identify a tagged partition $(\D,\tau)$ with the 
set $\{(Q,x)\colon Q\in\D,\;x=\tau(Q)\}$.
By a \textit{gage} we mean a strictly positive function $\delta:\rn\to\er$.
Given a gage $\delta$,
we say that a tagged partition $(\D,\tau)$ of $I$ is \textit{$\delta$-fine}
if for each $(Q,x)\in (\D,\tau)$ we have $\diam Q<\delta(x)$.
Let $f:I\to\er$ be a function and $\alpha\in \er$. We say that $\alpha$
is a $H\!K$ (Henstock-Kurzweil version of Stieltjes) integral
of $f$ over $I$ with respect to $\bG$ if for each $\ep>0$ there
exists a gage $\delta$ such that for each $\delta$-fine partition $(\D,\tau)$
of $I$ we have
\eqn{HK}
$$
\Bigl|\sum_{(Q,x)\in (\D,\tau)}f(x)\bG(Q)-\alpha\Bigr|<\ep.
$$
The $H\!K$-integral is unique if it exists.
We define the \textit{indefinite $H\!K$-integral} of $f:\rn\to\er$ as the 
interval function which assign to each $Q\in\I$ the 
$H\!K$-integral
of $f$ over $Q$ with respect to $\bG$.
It is an additive interval function.
\end{definition}

\begin{definition} Let 
Let $\Psi:\I\times \rn\to\er$ be a function and $\delta$ 
be a gage. 
The $\delta$-variation of $\Psi$ is defined as
$$
V_{\delta}(P,\Psi)=\sup\Bigl\{\sum_{(Q,x)\in (\D,\tau)}|\Psi(Q,x)|\colon
\text{$(\D,\tau)$ is a $\delta$-fine partition of $P$}\Bigr\},\quad P\in \I.
$$
If the interval function $V_{\delta}(\cdot,\Psi)$ is finite, then 
it is superadditive.
\end{definition}

The following statement establishes the equivalence of the Henstock-Kurzweil
integral and the so-called variational integral. The only if part is known as
Henstock's lemma. For the proof see e.g.\ \cite{He15}, Theorem 44.6,
\cite{He2}. 

\begin{proposition}
\label{p:H}
Let $\bF$, $\bG$ be additive interval functions on $\I$, $\bG\ge 0$.
Let $f:\rn\to\er$ be a function.
Let $\Psi:\I\times \rn\to\er$ is defined by
$$
\Psi(Q,x)= f(x)\bG(Q)-\bF(Q),\qquad Q\in\I,\ x\in \rn.
$$
Then $\bF$ is an indefinite
$H\!K$ integral of $f$ with respect to $\bG$ if and only if
$$
\inf\Bigl\{V_{\delta}(I,\Psi)\colon\text{$\delta$ is a gage}\Bigr\}=0
$$
for each $I\in\I$.
\end{proposition}

Now, we are ready to compare our $MC$-definition of integral with the
$H\!K$-integral.

\begin{theorem} Let  $f:\rn\to\er$ be a function and $\bF:\I\to\er$
be an interval function. Then $\bF$ is an indefinite $MC$-integral of
$f$ with respect to $\bG$ if and only if $\bF$ is an indefinite 
$H\!K$-integral of $f$ with respect to $\bG$.
\end{theorem}

\begin{proof} Let $\bPhi:\I\to\er$ be a control function to $(\bF,f)$ with
respect to $\bG$. Choose $I\in\I$ and $\ep>0$. For each $x\in I$ there exists
$\delta(x)>0$ such that for all $Q\in \I$ containing $x$ with $\diam Q<\delta$
we have
$$
\Bigl|\bF(Q)-f(x)\bG(Q)\Bigr|<\ep \bPhi(Q).
$$
We claim that $\delta$ is the desired gage. If $(\D,\tau)$ is a $\delta$-fine 
partition of an interval $I\in\I$, then 
$$
\aligned
\Bigl|\sum_{(Q,x)\in (\D,\tau)}f(x)\bG(Q)-\bF(I)\Bigr|&=
\Bigl|\sum_{(Q,x)\in (\D,\tau)}(f(x)\bG(Q)-\bF(Q))\Bigr|
\\&\le \sum_{(Q,x)\in (\D,\tau)}\Bigl|(f(x)\bG(Q)-\bF(Q))\Bigr|
\\&\le \sum_{(Q,x)\in (\D,\tau)}\ep\bPhi(Q)<\ep\bPhi(I).
\endaligned
$$
It follows that $\bF$ is an indefinite $H\!K$-integral of $f$ with respect to $\bG$.
Conversely, supppose
that $\bF$ is an indefinite $H\!K$-integral of $f$ with respect to $\bG$ and
denote 
$$
\Psi(Q,x)= f(x)\bG(Q)-\bF(Q),\qquad (Q,x)\in \I\times\rn.
$$
We consider the intervals $I_k=[-k,k]^n$. 
Using the Henstock lemma (Proposition \ref{p:H}),
for each  integer $k>0$ we find a gage $\delta_k$ on $I_k$ such that 
$$
V_{\delta_k}(I_k,\Psi)\le 2^{-k} .
$$
Set
$$
\bPhi(Q)=|Q|+\sum_{k=1}^{\infty} k\, V_{\delta_k}(Q\cap I_k). 
$$
Then $\bPhi(Q)$ is a strictly positive finite superadditive function on $\I$.
Given $x\in\er$ and $\ep>0$ we find an integer $k>1$ such that $\frac1k<\ep$
and $x\in I_{k-1}$. Let $Q\in\I$ be such that $x\in Q\subset I_k$ and 
$\diam Q<\delta_k(x)$. Then
$$
 |f(x)\bG(Q)-\bF(Q)|=|\Psi(x,Q)|\le V_{\delta_k}(Q,\Psi)\le 
 \frac1k\bPhi(Q)<\ep\bPhi(Q).
$$
This shows that $\bF$ is an indefinite $MC$-integral of $f$ with respect
to $\bG$ controlled by $\bPhi$.
\end{proof}

\subsection*{Acknowledgment} We thank Lud\v{e}k Zaj\'{\i}\v{c}ek for
valuable discussion and comments.

\end{document}